\documentclass[12pt]{article}

\usepackage[all]{xypic}
\usepackage{amsmath}
\usepackage{amsfonts}
\usepackage{amssymb}
\usepackage{amsthm}

\usepackage[top=2cm, bottom=2cm, left=2cm, right=2cm]{geometry}

\newtheorem{thm}{Theorem}[section]

\newtheorem{crl}[thm]{Corollary}

\newtheorem{prp}[thm]{Proposition}

\newtheorem{lmm}[thm]{Lemma}

\newtheorem{conj}[thm]{Conjecture}

\newtheorem{rmk}[thm]{Remark}

\newcommand {\mb}{\mathbb}
\newcommand {\Z}{\mb Z}

\newcommand {\colim}{\textrm{colim}\ }

\newcommand {\ex}{\mathrm{excess}}
\newcommand {\ext}{\mathrm{Ext}}
\newcommand {\lra}{\longrightarrow}
\newcommand {\la}{\langle}
\newcommand {\ra}{\rangle}

\begin{document}

\title{Spherical classes in some finite loop spaces of spheres}

\author{Hadi Zare\\
        School of Mathematics, Statistics,
        and Computer Sciences\\
        University of Tehran, Tehran, Iran\\
        \textit{email:hadi.zare} at \textit{ut.ac.ir}}
\date{}

\maketitle

\begin{abstract}
Working at the prime $2$, Curtis conjecture predicts that, in positive dimensions, spherical classes in $H_*QS^0$ only arise from Hopf invariant one and Kervaire invariant one elements. Eccles conjecture states that, in positive dimensional, for a path connected space $X$, a class in $H_nQX$ is spherical if its either stably spherical or it arises from a stable map $S^n\to X$ which is detected by a primary operation in its mapping cone.\\
(i) We use Hopf invariant one result to verify Eccles conjecture on some finite loop spaces of spheres, namely, for $X=S^k$ , we completely determine spherical classes in $H_*(\Omega^dS^{k+d};\Z/2)$ for specific values of $d$, $k$ showing that a spherical classes in these cases only do arise from Hopf invariant one elements.\\
(ii) We completely determine spherical classes in homology of single, double, and triple loop spaces of spheres, namely $\Omega S^{n+1}$, $\Omega^2 S^{n+2}$, and $\Omega^3S^{n+3}$ with $n>0$.\\
These computations, verify Eccles conjecture on the finite loop spaces that we have considered. We also record some observations on the relation between two conjectures. The latter conjecture for $X=S^k$, with $k>0$, provides some evidence for the former to be true. \\
\end{abstract}

\textbf{AMS subject classification:$55Q45,55P42$}

\tableofcontents

\section{Introduction}
For a pointed space $X$, let $QX=\colim\Omega^i\Sigma^iX$ be the infinite loop space associated to $\Sigma^\infty X$. Throughout the paper, we shall work at the prime $2$, the homology will be $\Z/2$-homology. We may state Curtis conjecture as follows.

\begin{conj}(\cite[Theorem 7.1]{Curtis})\label{Curtisconj}
For $n>0$, only Hopf invariant one and Kervaire invariant one elements map nontrivially under unstable Hurewicz homomorphism
$${_2\pi_n^s}\simeq{_2\pi_n}QS^0\to H_*QS^0.$$
\end{conj}

Here and throughout, we write ${_2\pi_*^s}$ and ${_2\pi_*}$ for the $2$-component of $\pi_*^s$ and $\pi_*$ respectively. We also write $H_*$ for $H_*(-;\Z/2)$. A variant of this conjecture is due to Eccles which may be stated as follows.

\begin{conj}\label{Ecclesconj}
(\textrm{Eccles conjecture}) Let $X$ be a path connected $CW$-complex with finitely generated homology. For $n>0$, suppose $h(f)\neq 0$ where ${_2\pi_n^s}X\simeq{_2\pi_n}QX\to H_*QX$ is the unstable Hurewicz homomorphism. Then, the stable adjoint of $f$ either is detected by homology or is detected by a primary operation in its mapping cone.
\end{conj}

Note that the stable adjoint of $f$ being detected by homology means that $h(f)\in H_*QX$ is stably spherical, i.e. it survives under homology suspension $H_*QX\to H_*X$ induced by $\Sigma^\infty QX\to \Sigma^\infty X$ given by the adjoint of the identity $QX\to QX$.

The above conjectures make predictions about the image of spherical classes in the image of unstable Hurewicz homomorphism $h:{_2\pi_*^s}X\simeq{_2\pi_*}QX\to H_*QX$. By Freduenthal supension theorem, for any $f\in\pi_nQX$, depending on the connectivity of $X$, we may find some nonnegative integer $i$ so that $f$ does pull back to $\pi_n\Omega^i\Sigma^iX$. So, it is natural to ask about the image of the Hurewicz homomorphism $\pi_n\Omega^i\Sigma^iX\to H_n\Omega^i\Sigma^iX$. Note that there exist obvious commutative diagrams
$$\xymatrix{
\pi_n\Omega^i\Sigma^iX\ar[r]^-h\ar[d] & H_n\Omega^i\Sigma^iX\ar[d]\\
\pi_nQX\ar[r]^-h                      & H_nQX.}$$
It is then natural to look for spherical classes in $H_*\Omega^i\Sigma^iX$ for all $i$ in order to give some answer to the above conjectures. As a possible strategy to answer Curtis conjecture, one may try to show that for any $f\in{_2\pi_n}QS^0$ other than Hopf or Kervaire invariant one elements, its pull back to ${_2\pi_n}\Omega^i\Sigma^iS^0$ maps trivially under the Hurewicz homomorphism. Similar approach may be taken to provide an affirmative answer to Eccles conjecture.

\section{Statement of results}
We use Hopf invariant one result, interpreted in terms of James-Hopf invariants, to determine spherical classes in $H_*\Omega^dS^{d+k}$ for specific values of $d,k$. For various values of $d$ and $k$, the spaces $\Omega^dS^{d+k}$ determine a lattice with positive integer coordinates, and it is easier to state our results in terms of this lattice.
\begin{thm}\label{main}
Consider the following lattice where $(a,b)$ corresponds to $H_*\Omega^aS^b$.
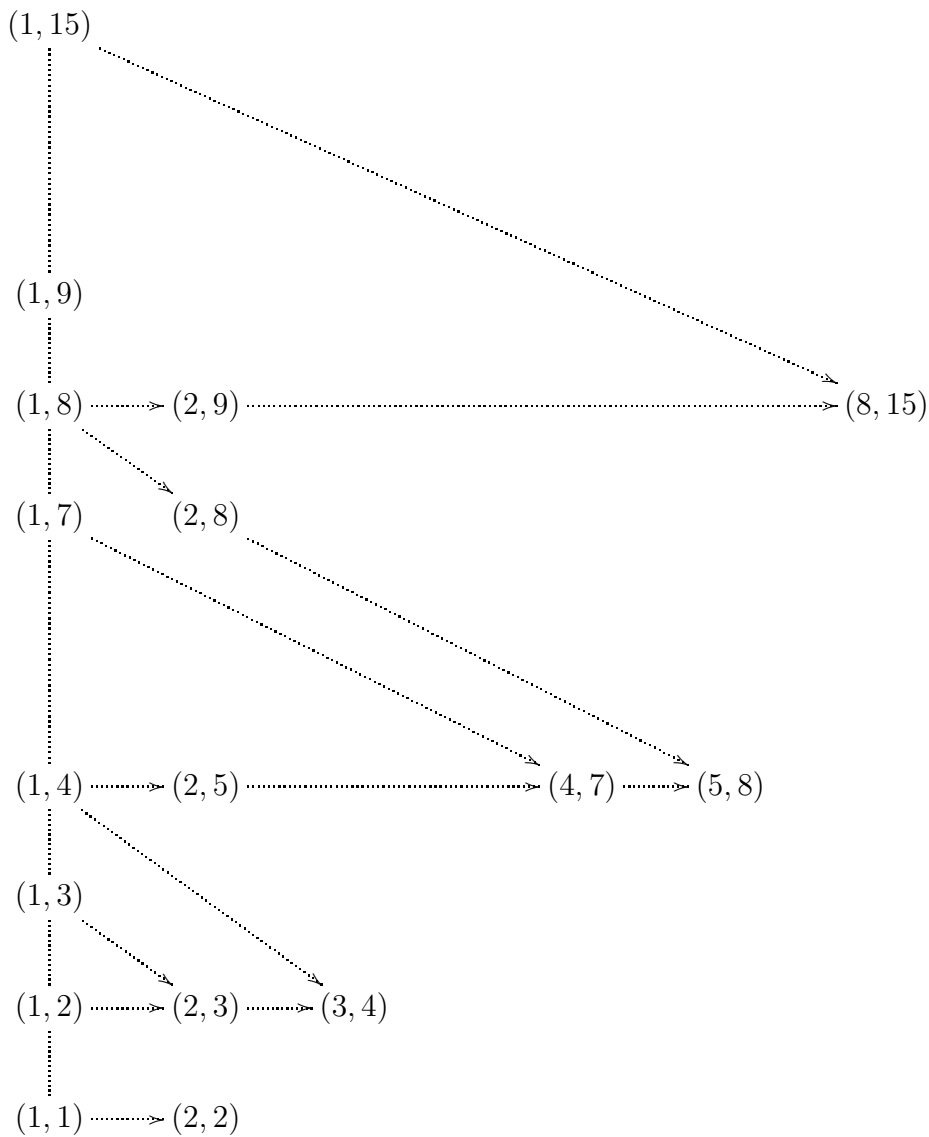
$$\xymatrix{
(1,15)\ar@{.}[ddd]\ar@{.>}[ddddrrrrrr]  \\
\\
\\
(1,9)\ar@{.}[d]\\
(1,8)\ar@{.>}[r]\ar@{.}[d]\ar@{.>}[rd] & (2,9)\ar@{.>}[rrrrr] & & & & & (8,15)\\
(1,7)\ar@{.>}[dddrrrr]\ar@{.}[ddd]     & (2,8)\ar@{.>}[rrrrddd]\\
\\
\\
(1,4)\ar@{.>}[r]\ar@{.}[d]\ar@{.>}[ddrr]          & (2,5)\ar@{.>}[rrr]&&& (4,7)\ar@{.>}[r] & (5,8)\\
(1,3)\ar@{.>}[rd]\ar@{.}[d]\\
(1,2)\ar@{.>}[r]\ar@{.}[d]& (2,3)\ar@{.>}[r] & (3,4)\\
(1,1)\ar@{.>}[r]          & (2,2)
}$$
The following table completely determines the spherical classes in the above lattice.
\begin{center}
\begin{tabular}{|c|c|c|c|c|c|c|c}
\hline
$b-a$ & $a$             & $b$              & \textrm{spherical classes arise from}\\
\hline
$0$   & $2$             & $2$              & $\eta,\eta^2$\\
\hline
$1$   & $2$             & $3$              & $\eta,b$\\
\hline
$1$   & $1$             & $2$              & $\eta,b$\\
\hline
$1$   & $3$             & $4$              & $\nu,\eta,b$\\
\hline
$2$   & $1$             & $3$              & $b$\\
\hline
$3$   & $\leqslant 4$   & $a+3$            & $\nu,b$\\
\hline
$3$   & $5$             & $8$              & $\sigma,\nu,b$\\
\hline
$7-a$ & $<4$            & $7$              & $b$\\
\hline
$7$   & $\leqslant 8$   & $a+7$            & $\sigma,b$\\
\hline
$15-a$& $<8$            & $15$            & $b$\\
\hline
$4-a$ & $<3$            & $4$              & $\nu,b$\\
\hline
$8-a$ & $<5$            & $8$              & $\sigma,b$\\
\hline
\textrm{other cases}    &    &                   & $b$\\
\hline
\end{tabular}
\end{center}
Here, $b$ corresponds to the inclusion of the bottom cell in the related loop space, that is $b:S^k\to \Omega^dS^{d+k}$. The diagonal arrows correspond to (iterated) adjointing down, and horizontal arrows correspond to (iterated) suspension map $E$. Both of this operations, preserve spherical classes.
\end{thm}


The stablisation map $\Omega^dS^{d+k}\to \Omega^{d+l}S^{d+k+l}$ induces a monomorphism in homology, hence the spherical classes computed in the above theorem survive under this map. However, having determined spherical classes in $H_*\Omega^dS^{d+k}$ does not completely determine spherical classes in $H_*\Omega^{d+l}S^{d+k+l}$; for some more discussions see Corollary \ref{Curtis-1}.\\

Next, using more homological tools coupled with James and James-Hopf maps, we completely determine spherical classes in homology of double loop spaces associated to spheres. Note that our previous result deals with spherical classes in homology of $\Omega S^{n+1}$ and $\Omega^2S^{n+2}$ for $n\leqslant 14$ and $n\leqslant 12$ respectively. Moreover, for $n\leqslant 10$, our results above determine all spherical classes in $H_*\Omega^3S^{n+3}$ expect for $\Omega^3S^5$ and $\Omega^3S^9$. The following theorem, deals with the other cases (some cases do overlap with the above calculations).

\begin{thm}
(i) For $n\geqslant 8$, the only spherical classes in $H_*\Omega S^{n+1}$ arise from the inclusion of the bottom cell $S^n\to\Omega S^{n+1}$.\\
(ii) For $n\geqslant 7$, the only spherical classes in $H_*\Omega^2 S^{n+2}$ arise from the inclusion of the bottom cell $S^n\to\Omega^2 S^{n+2}$.\\
(iii) For $n\geqslant 6$, the only spherical classes in $H_*\Omega^3 S^{n+3}$ arise from the inclusion of the bottom cell $S^n\to\Omega^3 S^{n+3}$. \\
(iv) The only spherical classes in $H_*\Omega^3S^5$ arise in dimensions $2$ and $5$ corresponding to the inclusion of the bottom cell and $\nu$ respectively.\\
(v) The only spherical classes in $H_*\Omega^3S^9$ arise in dimensions $6$ and $13$ corresponding to the inclusion of the bottom cell and $\sigma$ respectively.
\end{thm}

\section{Recollection on homology}
The homology of iterated loop spaces of the form $\Omega^i\Sigma^i S^n$ and $QS^n$ with $n\geqslant 0$ is well known (see for example
\cite[Theorem 3]{KudoAraki},\cite[Theorem 7.1]{KudoAraki-Hn}, \cite[Page 86, Corollary 2]{DyerLashof}, \cite{CLM}). The following is the only fact that we need to know about homology.
\begin{prp}\label{homology-injec}
For $n\geqslant 0$ and $i>0$, the stablisation map $\Omega^i_0\Sigma^iS^n\to Q_0S^n$ induces a monomorphism in homology. Here, the subindex $0$ denotes the base point component of the related loop spaces.
\end{prp}

Let's note that for $i=1$, and an arbitrary space $X$, $H_*\Omega\Sigma X\to H_*QX$ may not be a monomorphism in general, as $H_*\Omega\Sigma X$ being a tensor algebra is not necessarily commutative. However, when $X=S^n$, $\widetilde{H}_*X$ has only one generator, and a tensor algebra over a single generator is commutative. The statement of the proposition, is immediate once we consider the description of homology. First, we present a description in terms of lower indexed Kudo-Araki operations so that hopefully convince the reader about this.\\

For an $i$-fold loop space $X$, the operation $Q_j$ is defined for $0\leqslant j<i$ as an additive homomorphism
$$Q_j:H_*X\to H_{2*+j}X$$
so that $Q_0$ is the same as squaring with respect to the Pontrjagin product on $H_*X$ coming from the loop sum on $X$. The homology rings $H_*\Omega^i\Sigma^i S^n$ and $H_*QS^n$ when $n>0$, as algebras, can be described as
$$\begin{array}{lll}
H_*\Omega^i\Sigma^iS^n &\simeq &\Z/2[Q_{j_1}\cdots Q_{j_r}g_n:j_1\leqslant j_2\leqslant\cdots\leqslant j_r<i,j_1<n],\\
H_*QS^n                &\simeq &\Z/2[Q_{j_1}\cdots Q_{j_r}g_n:j_1\leqslant j_2\leqslant\cdots\leqslant j_r,j_1<n].
\end{array}$$
In this description, we allow the empty sequence $\phi$ as a nondecreasing sequence of nonnegative integers with $Q_\phi$ acting as the identity; this in fact realises the monomorphism $H_*S^n\to H_*\Omega^i\Sigma^iS^n$ given by the inclusion of the bottom cell $S^n\to\Omega^i\Sigma^iS^n$ being adjoint to the identity $S^{i+n}\to S^{i+n}$.
\begin{proof}[Proof of Proposition \ref{homology-injec} in the case of $n>0$]
Let $S^{i+n}\to QS^{i+n}$ be the inclusion of the bottom cell. By the infinite loop space structure of $QS^{i+n}$, it is also an $i$-fold loop map. By looping this map $i$-times we have the stablisation map $\Omega^i\Sigma^i S^n\to \Omega^iQS^{i+n}=QS^n$. In homology, the map obviously commutes with the operations coming from the $i$-loop structure, hence sending $Q_{j_1}\cdots Q_{j_r}g_n$ identically to itself. The maps is a map of algebras, so being injection on monomials, it induces an injection on the homology.
\end{proof}

For the case $n=0$, following Wellington \cite{Wellington} (see also the discussion \cite[Section 2]{Hunter} as well as \cite[Theorem 2.4]{Hunter}) we may describe $H_*\Omega^{i+1}_0S^{i+1}$ as follows. Let $M_iS^0$ be the free $\Z/2$-module generated by symbols $Q_J\iota_0$ with $J=(j_1,\ldots,j_r)$ a nonempty and nondecreasing sequence of nonnegative integers with $j_r<i+1$ so that $J\neq (0,0,\ldots,0)$. Write $[n]$ for the image of $n\in\pi_0\Omega^{i+1}S^{i+1}$ in $H_0\Omega^{i+1}S^{i+1}$ under the Hurewicz map. There is an embedding $M_iS^0\to H_*\Omega^{i+1}S^{i+1}$ which sends $Q_J\iota_0$ to $Q_J[1]*[-2^{l(J)}]$ where $l(J)=r$. We then have the following.
\begin{lmm}(Wellington)
For $i>0$, $H_*\Omega^{i+1}_0S^{i+1}$ is isomorphic to the free commutative algebra generated by $M_iS^0$ modulo the ideal generated by $\{Q_0x-x^2\}$.
\end{lmm}
This allows to think of $H_*\Omega^{i+1}_0S^{i+1}$ as
$$\Z/2[Q_{j_1}\cdots Q_{j_r}[1]*[-2^{l(J)}]:0\leqslant j_1\leqslant j_2\leqslant\cdots\leqslant j_r<i+1]$$
where we do not allow the empty sequence, but we allow $Q_0$ acting as the squaring operation. For $Q_0S^0$ being the base point component of $QS^0$, we have
\cite[Page 86, Corollary 2]{DyerLashof} (see also \cite[Part I, Lemma 4.10]{CLM})
$$H_*Q_0S^0\simeq\Z/2[Q_J[1]*[-2^{l(J)}]:0\leqslant j_1\leqslant j_2\leqslant\cdots\leqslant j_r]$$
where we do not allow the empty sequence, but we allow $Q_0$ acting as the squaring operation. Note that for $J$ being nondecreasing in $Q_Jx$ is the same as $I$ being admissible when we translate $Q_Jx=Q^Ix$ in terms of upper indexed operations. Again, as the above proof, we have an obvious monomorphism of algebras
$$H_*\Omega^{i+1}_0S^{i+1}\to H_*Q_0S^0$$
provided by the stablisation map $\Omega^{i+1}_0S^{i+1}\to Q_0S^0$.\\

\textbf{Homology in terms of $Q^i$ operations.} For some purposes, it is easier to use description of $QX$ in terms of the operations $Q^i$. For an infinite loop space $Y$ these are additive homomorphisms $Q^i:H_*Y\to H_{*+i}Y$ operations which relate to lower indexed operations by $Q_iz=Q^{i+d}z$ for any $d$-dimensional homology class $z$. The homology of $QX$, as a module over the Dyer-Lashof algebra, is determined by
$$H_*QX\simeq\Z/2[Q^Ix_\mu:I\textrm{ admissible },\ex(I)>\dim x_\mu]$$
where $I=(i_1,\ldots,i_s)$ is admissible if $i_j\leq 2i_{j+1}$, $\ex(I)=i_1-(i_2+\cdots+i_s)$, and $\{x_\mu\}$ is an additive basis for $\widetilde{H}_*X$. We allow the empty sequence to be admissible with $\ex(\phi)=+\infty$ and $Q^\phi$ acting as the identity. In particular, in this description, $Q^iz=0$ if $i<\dim z$ and $Q^dz=z^2$ if $d=\dim z$.\\
The action of the Steenrod algebra on $H_*QX$ is determined by Nishida relations which read as follows
$$Sq_*^rQ^a=\sum_{t\geqslant 0}{a-r\choose r-2t}Q^{a-r+t}Sq^{t}_*.$$
Note that the maximum value for $t$ is $[r/2]+1$. In particular, we have
$$Sq^1_*Q^{2d}=Q^{2d-1},\ Sq^{1}_*Q^{2t+1}=0.$$
We also have the Cartan formula $Sq^{2t}_*\eta^2=(Sq^t_*\eta)^2$ \cite{Wellington}.

\section{Main results and computations}
Previously, we have determined the image of the unstable Hurewicz map ${_2\pi_*^s}\simeq{_2\pi_*}QS^0\to H_*QS^0$ when restricted to the submodule of decomposable elements in ${_2\pi_*^s}$. More precisely, we have the following \cite[Theorem 1.2]{Z-decomposables} (see also \cite[Theorem 3]{Za-ideal}).

\begin{thm}\label{decomp}
$(i)$ For $i,j>0$, consider the composition
$${\pi_i}QS^0\otimes {\pi_j}QS^0\lra {\pi_{i+j}}QS^0\stackrel{h}{\lra} H_{i+j}(QS^0;\Z)$$
where the first arrow is the product in $\pi_*^s$. If $i\neq j$ then $h(fg)=0$. Moreover, if $h(fg)\neq 0$ then $i=j$ and $f,g$ both are detected by the unstable Hopf invariant.\\
$(ii)$ For $i,j>0$, consider the composition
$${_2\pi_i}QS^0\otimes {_2\pi_j}QS^0\lra {_2\pi_{i+j}}QS^0\stackrel{h}{\lra} H_{i+j}QS^0$$
where the first arrow is the product in $\pi_*^s$. If $i\neq j$ then $h(fg)=0$. Moreover, if $h(fg)\neq 0$ then $i=j$ and $f=g$ with $f=\eta,\nu,\sigma$ or odd multiples of these elements. Also, the image of the composite
$$\la f:f=\eta,\nu,\sigma\ra\rightarrowtail{_2\pi_*}QS^0\stackrel{h}{\lra}H_*QS^0$$
only consists of the Hurewicz image of the Hopf invariant one elements $\eta,\nu,\sigma$, and the Kervaire invariant one elements $\eta^2,\nu^2,\sigma^2$.\\
\end{thm}

Before going to our results, we recall some facts and fix our notation. We consider using James fibrations $S^n\stackrel{E}{\to}\Omega S^{n+1}\stackrel{H}{\to}\Omega S^{2n+1}$ referring to $E$ as the suspension and $H$ as the second James-Hopf map. We also write $E^\infty:X\to QX$ for the inclusion which induces $\pi_*X\to\pi_*QX\simeq\pi_*^sX$ sometimes referring to it as the stablisation. We also may write $E^k$ for the iterated suspension $X\to\Omega^k\Sigma^kX$. We shall use $H_*\Omega\Sigma X\simeq T(\widetilde{H}_*X)$ where $X$ is any path connected space, $T(-)$ is the tensor algebra functor and $\widetilde{H}_*$ denotes the reduced homology. We shall write $\sigma_*:H_*\Omega X\to H_{*+1}X$ for homology suspension induced by the evaluation $\Sigma\Omega X\to X$ recalling that it kills decomposable elements in the Pontrjagin ring $H_*\Omega X$.\\

We begin with the following.

\begin{thm}
(i) The only spherical classes in $H_*\Omega^2S^3$ live in dimensions $1$ and $2$, arising from the identity $S^3\to S^3$ and $\Sigma\eta:S^4\to S^3$. There are no spherical classes in $H_*\Omega S^3$ other than the bottom dimensional class given  by $S^2\to\Omega S^3$.\\
(ii) The only spherical classes in $H_*\Omega^4S^7$ live in dimensions $3$ and $6$, arising from the identity $S^7\to S^7$ and $\Sigma^3\nu:S^{10}\to S^7$. There are no spherical classes in $H_*\Omega^i S^7$ for $i<4$ other than $S^{3+i}\to \Omega^{4-i}S^7$ adjoint to the identity $S^7\to S^7$ that corresponds to the inclusion of the bottom cell.\\
(iii) The only spherical classes in $H_*\Omega^8S^{15}$ live in dimensions $7$ and $14$, arising from the identity $S^{15}\to S^{15}$ and $\Sigma^7\sigma:S^{22}\to S^{15}$. There are no spherical classes in $H_*\Omega^i S^{15}$ for $i<8$ other that the one arising from the inclusion of the bottom cell corresponding to the identity on $S^{15}$.\\
\end{thm}

\begin{proof}
Let $k\geqslant 0$. Consider the stablisation $S^{i+k}\to QS^{i+k}$ and by abuse of notation write $E^\infty:\Omega^iS^{i+k}\to\Omega^iQS^{i+k}=QS^k$ for its $i$-fold loop. Recall that $E^\infty_*$ induces a monomorphism in homology. We also write $E:S^k\to\Omega S^{k+1}$ for the suspension map, that induces the suspension homomorphism $\pi_*S^k\to\pi_{*}\Omega S^{k+1}\simeq\pi_{*+1}S^{k+1}$. We proceed as follows.\\
(i) Obviously, the inclusion of the bottom cell $S^1\to\Omega^2S^3$ is nontrivial in homology. Let $f:S^n\to\Omega^2S^3$, $n>1$, be any map with $f_*\neq 0$. First, note that the Hopf fibration $\eta:S^3\to S^2$ induces a homotopy equivalence $\Omega^2\eta:\Omega^2S^3\to\Omega^2_0S^2$ where $\Omega^2_0S^2$ is the base point component of $\Omega^2S^2$. In particular, $\Omega^2\eta$ is an isomorphism in homology. Consequently, the composition
$((\Omega^2E^\infty)(\Omega^2\eta)f)_*\neq 0$. On the other hand, note that we have a commutative diagram
$$\xymatrix{
\Omega^2S^3\ar[r]^-{\Omega^2 E^\infty}\ar[d]^-{\Omega^2\eta}    & QS^1\ar[d]^-{\Omega^2 \eta}\\
\Omega^2_0S^2\ar[r]^-{\Omega^2 E^\infty}                        & Q_0S^0
}$$
where the right vertical map is obtained from $\eta:QS^3\to QS^2$ and we write $Q_0S^0$ for the base point component of $QS^0$. This implies that $(\Omega^2\eta)_*((\Omega^2E^\infty)f)_*\neq 0$. This implies that for $E^\infty f$ and $\eta$ as elements of $\pi_*^s$, we have $h(\eta(E^\infty f))\neq 0$ where $h$ is the unstable Hurewicz homomorphism. By Theorem \ref{decomp}(ii) if $n>0$ then $E^\infty f$ and $\eta$ are in the same dimension. That is $E^\infty f\in\pi_1^s$, consequently $n=2$. Therefore, $f\in\pi_2\Omega^2S^3\simeq\pi_4S^3\simeq\Z/2$, which implies that $f= E\eta$ viewing $\eta:S^3\to S^2$. Moreover, note that by Freudenthal's theorem $f$ pulls back to $\pi_3S^2$ allowing to consider $f$ as $f:S^2\to\Omega S^2$. The only $2$-dimensional class in $H_2\Omega S^2$ is given by $x_1^2$ which maps to $x_1^2\in H_2\Omega^2S^3$ which is a decomposable class, where $x_1\in\widetilde{H}_1S^1$ is a generator.\\
Moreover, suppose $f$ is any map $S^{n+1}\to\Omega S^3$ which is nontrivial in homology. By adjointing down, we have a map $f':S^n\to\Omega^2S^3$ which is nontrivial in homology. We have $f_*\sigma_*x=\sigma_*f'_*x$ for any homology class $x$. However, by the above computation, if $f'_*x\neq 0$ then $f'_*x=x_1^2$ which is a decomposable class that is killed under homology suspension. This contradicts the claim that $f_*\neq 0$ for $f:S^{n+1}\to\Omega S^3$. \\
(ii) The argument here is almost the same. We consider the Hopf fibration
$$\Omega S^7\to\Omega S^4\to S^3\to S^7\to S^4$$
noting that $H(\nu)=1$ where $H:\Omega S^4\to \Omega S^7$ is the second James-Hopf map. This provides us with a decomposition $\Omega S^4\to S^3\times\Omega S^7$. A choice of decomposition may be given by $\Omega S^7\stackrel{(*,1)}{\longrightarrow}S^3\times\Omega S^7 \stackrel{E+\Omega\nu}{\longrightarrow}\Omega S^4$ and $H:\Omega S^4\to \Omega S^7$ in the other direction where $+$ is the loop sum in $\Omega S^4$. The decomposition of $3$-fold loop spaces, in particular implies that there is monomorphism in homology given by $\Omega^4\nu:\Omega^4 S^7\to\Omega^4S^4$.\\
Note that $\widetilde{H}_n\Omega^4S^7\simeq 0$ for $n<3$. Similar to the previous part, the inclusion of the bottom cell $S^3\to\Omega^4S^7$, adjoint to the identity $S^7\to S^7$ is nontrivial in homology. Suppose $f:S^n\to\Omega^4S^7$ with $n>3$ and $f_*\neq 0$ is given. Then, similar as above, we have $((\Omega^3 E^\infty)(\Omega^3\nu)f)_*\neq0$. Consider the commutative diagram
$$\xymatrix{
S^7\ar[r]^-{E^\infty}\ar[d]^-\nu & QS^7\ar[d]^-{\nu}\\
S^4\ar[r]^-{E^\infty}            & QS^4}$$
and then loop it $4$-times. Moreover, using the splitting of $\Omega S^4$ as above, we may rewrite the quadruple loop of the left vertical arrow and obtain a commutative diagram of $4$-fold loop maps as
$$\xymatrix{
\Omega^4S^7\ar[r]^-{\Omega^4E^\infty}\ar[d]_-{\Omega^4(*,1_{S^7})} & QS^3\ar[dd]^-{\Omega^4\nu}\\
S^3\times\Omega^4S^7\ar[d]_-{\Omega^3(E+\Omega\nu)} \\
\Omega^4S^4\ar[r]^-{\Omega^4E^\infty}                     & QS^0}$$
noting that the composition of vertical arrows on the left is just $\Omega^4\nu$. By commutativity of the above diagram, we see that $h(\nu E^\infty f)\neq 0$. Consequently, by Theorem \ref{decomp}(ii) as elements of $\pi_*^s$, $E^\infty f$ is in the same dimension as $\nu$ implying that $n=6$. Therefore, $f\in{_2\pi_6}\Omega^4S^7\simeq{_2\pi_{10}}S^7\simeq\Z/8$ and $f=E^3\nu$ where $E^3$ is the iterated suspension (strictly speaking, $f$ will be three fold suspension of an odd multiple of $\nu$ in ${_2\pi_3^s}$, but since we are working at the prime $2$, we take it as $\nu$). By Freudenthal suspension theorem it pulls back to $\pi_7S^4\simeq\pi_6\Omega S^4$. Hence, $f$ does factorise as $f:S^6\to\Omega S^4\to\Omega^4S^7$. The map $S^6\to\Omega S^4$ has square of a three dimensional class as its Hurewicz image. Hence, as previous part, eliminating classes in $H_*\Omega^iS^7$ with $i<4$ from being spherical.\\
(iii) Similar to the previous part, one has to use a decomposition $\Omega S^8\simeq S^7\times\Omega S^{15}$. We leave the rest to the reader.\\
\end{proof}

The above has an application to the finite case of Curtis conjecture.

\begin{crl}\label{Curtis-1}
Suppose $f\in H_*\Omega^8S^8$ is a spherical class. Then, $f$ either corresponds to a spherical class in $\Omega^8S^{15}$ or is suspension of a spherical class in $\Omega^7S^7$ under stablisation map $\Omega^7S^7\to\Omega^8S^8$. In particular, there are spherical classes arising from Hopf invariant one elements $\eta,\nu,\sigma$ and Kervaire invariant one elements $\eta^2,\nu^2,\sigma^2$ in $H_*\Omega^8S^8$.
\end{crl}

Note that spherical classes in $\Omega^8S^{15}$ arise from $1$ and $\sigma$ which map to $\sigma$ and $\sigma^2$ under $\Omega\sigma:\Omega S^{15}\to\Omega S^8$. By previous theorem spherical classes in $\Omega^8S^{15}$ are known. Hence, the computations in the above theorem will be complete once we know spherical classes in $\Omega^7S^7$.

\begin{rmk}
Let's note that the existence of Hopf fibrations, as well as James-Hopf maps together with Freudenthal suspension theorem all work well integrally. Also, the computations of Hurewivz map on decomposable elements maybe carried out integrally as done in Theorem \ref{decomp}(i). We then conclude that the above results may be stated integrally, possibly with some modifications due to omitting the odd torsion.
\end{rmk}

\begin{crl}
(i) The only spherical class in $H_*\Omega S^2$ corresponds to $\eta\in\pi_1^s\simeq\pi_2QS^1$. In fact there is such a class given by $S^2\to\Omega S^2$ provided by the adjoint of $\eta:S^3\to S^2$.\\
(ii) The only spherical class in $H_*\Omega^iS^{3+i}$, $1\leqslant i\leqslant3$, corresponds to $\nu\in\pi_3^s\simeq\pi_6QS^3$. In fact there is such a class give by $S^6\to\Omega S^4\stackrel{E^{i-1}}{\to}\Omega^iS^{4+i}$. Moreover, there is no spherical class in $H_*\Omega^jS^{4+i}$ for $j<i$.\\
(iii) The only spherical class in $H_*\Omega^iS^{7+i}$, $1\leqslant i\leqslant7$, corresponds to $\sigma\in\pi_7^s\simeq\pi_{14}QS^7$. In fact there is such a class given by
$S^{14}\to\Omega S^8\stackrel{E^{i-1}}{\to}\Omega^iS^{8+i}$. Moreover, there is no spherical class in $H_*\Omega^jS^{8+i}$ for $j<i$.
\end{crl}

\begin{proof}
All of these follow from the above theorem together with the fact that the suspension map $\Omega S^{k+1}\to\Omega^{i}S^{k+i}$ induces a monomorphism in homology, hence preserving spherical classes. The nonexistence part, follows similar to the previous theorem by adjointing down. We do (ii) for illustration. Consider the iterated suspension map $\Omega^iS^{3+i}\to\Omega^4S^7$. Hence, for $f:S^n\to\Omega^iS^{3+i}$ with $f_*\neq 0$ the composition $S^n\to\Omega^iS^{4+i}\to\Omega^4S^7$ is nontrivial in homology. Therefore, $n=6$ and $f:S^6\to\Omega^iS^{i+3}$ which maps to $\nu$ by the above theorem. In fact, for $i>1$, $f\in\pi_{6+i}S^{3+i}\simeq\pi_3^s\simeq\Z/8$. Hence, $f=E^i\nu$. For $i=1$, $f:S^6\to\Omega S^4$ is nontrivial, if and only if $f_*=g_3^2$, hence the adjoint of $f$ is detected by the unstable Hopf invariant, which working at the prime $2$ means it is detected by $Sq^4$ in its mapping cone. This means $f$ is an odd multiple of $S^7\to S^4$.
\end{proof}

Finally, we conclude by computing spherical classes in $H_*\Omega^2S^2$.
\begin{thm}
The only spherical classes in $H_*\Omega^2S^2$ arise from the Hopf invariant one and Kervaire invariant one elements $\eta$ and $\eta^2$, respectively.\\
\end{thm}

\begin{proof}
We apply the homotopy equivalence $\Omega^2\eta$. By this equivalence, any spherical class in $H_*\Omega^2S^2$ is image of a spherical class in $H_*\Omega^2S^3$. By the above computations, spherical classes in $H_*\Omega^2S^3$ arise from $1,\eta\in\pi_*\Omega^2S^3$ which map to $\eta$ and $\eta^2$ in $\pi_*\Omega^2S^2$ under $(\Omega^2\eta)_\#:\pi_*\Omega^2S^3\to\pi_*\Omega^2S^3$. This completes the proof.\\
\end{proof}

\section{Further computations}
The aim of this section is to compute spherical classes in more loop spaces associated to spheres. For instance, the decomposition $\Omega S^4\simeq S^3\times\Omega S^7$ provides a decomposition of $2$-loop spaces $\Omega^3S^4\simeq \Omega^2S^3\times\Omega^3S^7$. The spherical classes in the second factor arise from inclusion of the bottom cell. Noting that the decomposition of $\Omega S^7$ off $\Omega S^4$ is provided by $\Omega\nu:\Omega S^7\to\Omega S^4$, the bottom cell of $\Omega S^7$ gives rise to a spherical class in $\Omega^3S^4$ detecte by $\nu$. The other spherical classes arise from the first factor, that is the image of $E:\Omega^2S^3\to\Omega^3S^4$. Moreover, note that by adjointing down, we may use a spherical classes in $\Omega^iS^4$, with $i<3$, to get a spherical classes in homology of $\Omega^3S^4$. Since Hurewicz image of $\eta:S^2\to\Omega S^2$ is a square, and since $\Omega E:\Omega S^2\to\Omega^3S^4$ is  loop map, then the Hurewicz image of $\eta:S^2\to\Omega^3S^4$ remains a square which will die under homology suspension $H_*\Omega^3S^4\to H_{*+1}\Omega^2S^4$. We then have proved the following.
\begin{lmm}
The only spherical classes in $H_*\Omega^3S^4$ arise from $\eta,\nu$ and the inclusion of the bottom cell $S^1\to\Omega^3S^4$. Moreover, the only spherical classes in $H_*\Omega^iS^4$ with $1\leqslant i<3$ arise from the inclusion of the bottom cell and $\nu$.
\end{lmm}
A similar reasoning as above, using the decomposition $\Omega S^8\simeq S^7\times\Omega S^{15}$, proves the following.
\begin{lmm}
The only spherical classes in $H_*\Omega^5S^8$ arise from $\sigma$, $\nu$ and the inclusion of the bottom cell $S^3\to\Omega^5S^8$. Moreover, the only spherical classes in $H_*\Omega^iS^4$ with $1\leqslant i<5$ arise from the inclusion of the bottom cell and $\sigma$.
\end{lmm}
Next, note that as the suspension $\Omega^3S^4\to\Omega^7S^8$ induces a monomorphism, hence we get spherical classes given by the inclusion of the bottom cell, $\eta$ and $\nu$. We also obtain, a spherical class given $\sigma$ by adjointing down the spherical class given by $\sigma$ in $\Omega^5S^8$.\\

\section{More homology}
It is possible to use more detailed description of homology in order to eliminate more classes from being spherical. A spherical class $\xi\in\widetilde{H}_*X$ has some basic properties: (1) it is primitive in the coalgebra $\widetilde{H}_*X$, (2) $Sq^i_*\xi=0$ for all $i>0$ where $Sq^i_*:\widetilde{H}_*X\to \widetilde{H}_{*-i}X$ is the operation induced by $Sq^i:\widetilde{H}^*X\to \widetilde{H}^{*+i}X$ by the vector space duality \cite[Lemma 6.2]{AsadiEccles}.

\subsection{Spherical classes in $H_*\Omega S^{n+1}$}
Our previous geometric arguments, completely determine spherical classes in $\Omega S^{n+1}$ for $n\leqslant 14$. The following theorem deals with the other cases.

\begin{thm}\label{singleloop}
For $n\geqslant 8$, the only spherical class in $H_*\Omega S^{n+1}$ is given by the inclusion of the bottom cell $S^n\to\Omega S^{n+1}$.
\end{thm}

\begin{proof}
The cases $n+1\leqslant 15$ is verified by previous computations. But, our method here works generally for all $n\geqslant 8$.\\

The inclusion of the bottom cell $S^n\to\Omega S^{n+1}$ is obviously nontrivial in homology. Suppose $S^i\to \Omega S^{n+1}$ with $i>n$. By James's description, $H_*\Omega S^{n+1}\simeq T(x_n)$ with $x_n\in H_n\Omega S^{n+1}$ coming from the inclusion of the bottom cell. For $i>n$, if $h(f)\neq 0$ then $h(f)=x_n^k$ is a decomposable. By \cite[Propostion 4.21]{MM} $h(f)$ as a primitive class which is a decomposable, must be a square, so $h(f)=x_n^{2k}$.\\
\textbf{Case $k=1$.} Working with $\Z/2$-homology, if $h(f)=x_n^2$ then its adjoint $S^{i+1}\to S^{n+1}$ has to be detected by the unstable $\Z/2$-Hopf invariant \cite[Proposition 6.1.5]{Harper}. By the Hopf invariant one result, this requires $n+1\in\{2,4,8\}$. This is not possible as $n\geqslant 8$.\\
\textbf{Case $k>1$.} An inductive argument can be employed to show that in a Hopf algebra over $\Z/2$ for a primitive element $x$, $x^i$ is primitive if and only if $i$ is a power of $2$. The element $x_n\in H_n\Omega S^{n+1}$ is primitive, so $h(f)=x_n^{2^t}$ for some $t>1$.\\
Compose $f$ with the stablisation map $E^\infty:\Omega S^{n+1}\to QS^n$. Then,
$$h(E^\infty f)=x_n^{2^t}=Q^{2^{t-1}n}\cdots Q^{2n}Q^nx_n.$$
Let $\widetilde{E^\infty f}:S^{i-1}\to QS^{n-1}$ be the adjoint of $E^\infty f$ with $h(\widetilde{E^\infty f})\neq 0$. By a diagram chase,
$$h(\widetilde{E^\infty f})=Q^{2^{t-1}n}\cdots Q^{2n}Q^nx_{n-1}+D$$
where $D$ is a sum of decomposable terms. It is known that for a primitive class $\xi$, $Q^i\xi$ is also primitive. Consequently, the first term in the above sum is primitive. This implies that $D$ also has to be primitive, and being decomposable it has to be a square, ie, a class of even dimension. But, $\widetilde{E^\infty f}$ is odd dimensional, hence $D=0$ which means $h(\widetilde{E^\infty f})=Q^{2^{t-1}n}Q^{2^{t-2}n}\cdots Q^{2n}Q^nx_{n-1}$ with being a class of dimension $2^{t-1}n-1$. By Nishida relations
$$Sq^1_*Q^{2^{t-1}n}Q^{2^{t-2}n}\cdots Q^{2n}Q^nx_{n-1}=Q^{2^{t-1}n-1}Q^{2^{t-2}n}\cdots Q^{2n}Q^nx_{n-1}=(Q^{2^{t-2}n}\cdots Q^{2n}Q^nx_{n-1})^2\neq 0.$$
But is in contradicts with the fact that $h(\widetilde{E^\infty f})$ has to be killed by $Sq^1_*$. This completes the proof.
\end{proof}

The method of eliminating $x_n^{2^t}$, with $t>1$, from being spherical can be used to prove a more general fact for a wider range of loop spaces. We record the following and leave the proof to the interested reader.

\begin{prp}\label{2^t}
Suppose $f:S^i\to\Omega^nS^{d+n}$ is given with $1\leqslant n\leqslant +\infty$ and $d\geqslant 1$ so that $h(f)=\xi^2$. Then, $\xi$ has to be a class of odd dimension. In particular, if $h(f)=x_d^{2^t}$ then $t=1$ and $x_d$ has to be odd dimension, ie, $d$ has to be odd.
\end{prp}

\subsection{Spherical classes in $H_*\Omega^2S^{n+2}$}

Similar to the case of $\Omega S^{n+1}$, we wish to show that beyond a range, it is impossible to have spherical classes in $H_*\Omega^2S^{n+2}$ other that the ones given by the inclusion of the bottom cell $S^n\to\Omega^2S^{n+2}$. Our result in this direction reads as the following.

\begin{thm}\label{doubleloop}
Suppose $n\geqslant 7$, then there is no spherical class in $H_*\Omega^2S^{n+2}$ other than the class given by the inclusion of the bottom cell $S^n\to\Omega^2S^{n+2}$.
\end{thm}

As usual, it is obvious that the inclusion of the bottom cell gives a spherical class. So, we focus on the other cases. We divide the proof into small lemmata.

\begin{lmm}\label{doubleloop-1}
Suppose $f:S^i\to\Omega^2S^{n+2}$ with $i>n$ and $n\geqslant 7$. Then, written in upper indexed Kudo-Araki operations, we have
$$h(f)=(Q^{n+1}x_n)^2.$$
Consequently, it is impossible to have such an $f$ if $n$ is odd.
\end{lmm}

\begin{proof}
For $n\geqslant 7$, by Lemma \ref{singleloop} the adjoint of $f:S^{i+1}\to\Omega S^{n+2}$ must have zero Hurewicz image. This implies that $h(f)\in H_*\Omega^2S^{n+2}$ must be a square. Applying Lemma \ref{2^t}, this implies that $h(f)$ has to be square of an odd dimensional class. Recall that
$$H_*\Omega^2S^{n+2}\simeq\Z/2[Q_{j_1}\cdots Q_{j_s}x_n:0\leqslant j_1\leqslant j_2\leqslant\cdots\leqslant j_s\leqslant 1]$$
where $Q_0$ is the squaring operation, and we may compute that in terms of the upper indexed operations, we have
$$\underbrace{Q_1\cdots Q_1}_{t-\textrm{times}}x_n=Q^{2^{s-1}(n+1)}\cdots Q^{4n+4}Q^{2n+2}Q^{n+1}x_n.$$
For $h(f)=\xi^2$ with $\xi\in H_*\Omega^2S^{n+2}$ being an odd dimensional class, we may write
$$\xi=\sum Q_Jx_n+D$$
where the sum ranges over certain sequences $J$ and $D$ is a sum of decomposable terms. Since $D$ is a primitive decomposable living in odd dimensions, so $D=0$. Since $\dim(\xi)=i>n$, hence $J$ is nonempty. According to our description, $J=(j_1,\ldots,j_s)$ with $0\leqslant j_1\leqslant\cdots\leqslant j_s\leqslant 1$. If, there is a $J$ with $j_1=0$, this implies that the inside of the bracket is of even dimension, which contradicts Lemma \ref{2^t}. Hence, the sum has to run over all sequences $J$ which is a sequences of $1$'s. If $l(J)=s>1$, this means that we have terms such as $Q^{2^{s-1}(n+1)}\cdots Q^{4n+4}Q^{2n+2}Q^{n+1}x_n$ as terms of $\xi$ where various $J$'s only differ by their length, that is
$$\xi=\sum_{i=1}^{l} Q^{2^{s_i-1}(n+1)}\cdots Q^{4n+4}Q^{2n+2}Q^{n+1}x_n$$
for some positive integer $l$ where $s_i$ is the length of a sequence $J$ appearing in the expression for $\xi$ and $i\neq j$ implies that $s_i\neq s_j$. Application of Nishida relations preserves length, by definition. Suppose, there is a $J_0$ with $l(J_0)>1$, and $Q_{J_0}x_n=Q^{2^{s-1}(n+1)}\cdots Q^{4n+4}Q^{2n+2}Q^{n+1}x_n$. Then, by Nishida relations $Sq^2_*\eta^2=(Sq^1_*\eta)^2$, compute that
$$Sq^2_*h(f)=Sq^2_*\xi^2=Sq^2_*(Q^{2^{s-1}(n+1)}\cdots Q^{4n+4}Q^{2n+2}Q^{n+1}x_n)^2+O^2$$
where $O$ is a sum of $Q_Jx_n$ with $J\neq J_0$ and consequently $l(J)\neq l(J_0)$. We then have
$$\begin{array}{lll}
Sq^2_*h(f) & = & (Sq^1_*Q^{2^{s-1}(n+1)}\cdots Q^{4n+4}Q^{2n+2}Q^{n+1}x_n)^2+(Sq^1_*O)^2\\
           & = & (Q^{2^{s-1}(n+1)-1}\cdots Q^{4n+4}Q^{2n+2}Q^{n+1}x_n)^2 + (Sq^1_*O)^2
           \end{array}$$
which noting that $\dim(Q^{2^{s-2}(n+1)}\cdots Q^{4n+4}Q^{2n+2}Q^{n+1}x_n)=2^{s-1}(n+1)-1$ implies that
$$Sq^2_*h(f)=(Q^{2^{s-2}(n+1)}\cdots Q^{4n+4}Q^{2n+2}Q^{n+1}x_n)^4 + (Sq^1_*O)^2$$
where the first term in the above expression is nonzero, and does not cancel out with any of terms $(Sq^1_*O)^2$ for the difference of length. Hence, $Sq^2_*h(f)\neq 0$ which is a contradiction. This implies that $J$ can only be of length one, so it is uniquely determined and $h(f)=(Q^{n+1}x_n)^2$.\\
Now, suppose $n$ is odd, that is $n+1$ is even. Applying the Nishida relation $Sq^2_*\eta^2=(Sq^1_*\eta)^2$ we compute that
$$Sq^2_*h(f)=(Q^{n}x_n)^2=x_n^4\neq 0.$$
This is a contradiction. So, $n$ must be even. This completes the proof.
\end{proof}

In order to complete the elimination of spherical classes in $H_*\Omega^2S^{n+2}$, we have to deal with the cases when $n$ is even.

\begin{lmm}\label{doubleloop-2}
Suppose $f:S^i\to\Omega^2S^{n+2}$ is given with $i>n$ and $n\geqslant7$, and $n$ even. Then, $h(f)=0$.
\end{lmm}

The proof will use various relations between stable and unstable James-Hopf invariants, together with an interpretation of the Hopf invariant one result. First, recall that by James splitting \cite{James-reducedproduct} for a path connected space $X$ we have a splitting $\Sigma\Omega\Sigma X\simeq\bigvee_{r=1}^{+\infty}\Sigma X^{\wedge r}$. By projection on the $r$-th summand and taking adjoint we have a map $H_r:\Omega\Sigma X\to \Omega\Sigma X^r$ which we call James-Hopf invariant; we write $H$ to denote $H_2$. Moreover, this map can be composed with the stablisation map to yield a map $\Omega \Sigma X\to Q X^{\wedge r}$. This generalises as follows to yield stable maps. Recall that, for $k\geqslant 1$, when $X$ is path connected, we have Snaith splitting \cite{Snaith}
$$\Sigma^\infty \Omega^k\Sigma^kX\simeq\bigvee_{r=1}^{+\infty}\Sigma^\infty D_r(X,k)$$
where $D_r(X,k)=E\Sigma_r\times_{\Sigma_r}X^{\wedge r}$, writing $D_rX=D_r(X,k)$. In particular, $D_2S^n\simeq\Sigma^n P_n$ \cite{Kuhngeometry} where $P^n$ denotes the $n$-dimensional projective space, $P$ is the infinite dimensional real projective space, and $P_n=P/P^{n-1}$. By projection on to the $r$-th summand, and taking stable adjoint, we have a map $j_r^k:\Omega^k\Sigma^k\to QD_r(X,k)$ which we refer to as the $r$-th James-Hopf map for $\Omega^k\Sigma^kX$, and for $k=+\infty$ we write $j_r:QX\to D_rX$ for this map and call it stable James-Hopf maps. The maps $j_{r}^{k+d}$ and $\Omega^k j_r^k$ are compatible by through the suspension $\Omega^k\Sigma^k X\to \Omega^{k+d}\Sigma^{k+d}X$ where $d>0$ fitting into some obvious commutative diagrams
\cite[Proposition 1.1, Theorem 1.2]{Kuhngeometry} (see also \cite{Milgram-unstable}).\\
The homology of these maps is also well understood \cite{Kuhnhomology}. Essentially, there is a filtration on $H_*\Omega^k\Sigma^kX$, called the hight filtration, and the stable projection $\Sigma^\infty\Omega^k\Sigma^kX\to \Sigma^\infty D_r(X,r)$ induces projection on the elements of height $r$ in homology, the the homology $j_r^k$ is compatible with this (see \cite{CLM} for example or \cite{BE3} for this); this is the only fact that we need about homology of these maps.\\
Finally, we note that by Boardmap and Steer detecting a map $S^{2n+1}\to S^{n+1}$ by $Sq^{n+1}$ in its mapping cone, corresponds to detecting the adjoint mapping $S^{2n}\to\Omega S^{n+1}$ by $H:\Omega S^{n+1}\to \Omega S^{2n+1}$ \cite{BoardmanSteer} and the latter is equivalent to detecting the adjoint map $\widetilde{f}:S^{2n}\to\Omega S^{n+1}$ in homology by $h(\widetilde{f})=x_n^2$ \cite[Proposition 6.1.5]{Harper}. A similar and more general statement holds on maps $f:S^{2n}\to QX$ saying that $h(f)=x_n^2$ with $x_n\in\widetilde{H}_nX$ then the stable adjoint of $f$, $S^n\to X$ is detected by $Sq^{n+1}x_n=x_{2n+1}$ in its stable mapping cone \cite[Proposition 5.8]{AsadiEccles}.

\begin{proof}[Proof of Lemma \ref{doubleloop-2}]
By Lemma \ref{doubleloop-1} we have
$$h(f)=(Q^{n+1}g_n)^2$$
in particular, $i=4n+2$. Consider the composition
$$S^{4n+2}\stackrel{f}{\to} \Omega^2S^{n+2}\stackrel{\Omega H}{\to}\Omega^2 S^{2n+3}$$
where $H:\Omega S^{n+2}\to \Omega S^{2n+3}$ is the James-Hopf invariant. We have a commutative diagram
$$\xymatrix{
\Omega^2S^{n+2}\ar[r]^{\Omega H}\ar[d]_{=}      & \Omega^2 S^{2n+3}\ar[d] \\
\Omega^2S^{n+2}\ar[r]^{j_2^2}\ar[d]_{\Omega^2E} & QD_2(S^{n+1},2)\ar[d]\\
QS^n\ar[r]^-{\Omega j_2}                        & Q\Sigma^n P_{n+1}}
$$
where $E:S^{n+2}\to QS^{n+2}$ is the stablisation map, and $j_2:QS^{n+1}\to QD_2S^{n+1}=Q\Sigma^{n+1}P_{n+1}$. The composition of the vertical arrows on the left is injective in homology. Moreover, the map $\Omega j_2$ is multiplicative, and $j_2:QS^{n+1}\to QD_2S^{n+1}$ sends elements of height $2$, to elements of $\widetilde{H_*}D_2S^{n+1}$. In particular, we have $(j_2)_*Q^{n+1}g_{n+1}=\Sigma^{n+1}g_{n+1}$ and for dimensional reasons there are no other terms in this expression. This implies that
$$(\Omega j_2)(Q^{n+1}x_n)^2=(\Sigma^na_{n+1})^2\in H_*Q\Sigma^n P_{n+1}$$
where $a_i\in\widetilde{H}_iP_{n+1}$ is a generator. By a diagram chase, we conclude that
$$(\Omega H)_*h(f)=(\Sigma^nx_{n+1})^2\in H_*\Omega^2S^{2n+3}.$$
Next, by a similar method, for $H:\Omega S^{2n+3}\to \Omega S^{4n+5}$, by a similar method we have
$$((\Omega H)\circ(\Omega H))_*h(f)=x_{4n+4}.$$
This implies that $H:\Omega S^{2n+3}\to \Omega S^{4n+5}$ detects the adjoint map $\widetilde{\Omega H\circ f}:S^{4n+5}\to S^{2n+3}$ which means that $H(f):S^{4n+5}\to S^{2n+3}$ has to be detected by $Sq^{2n+3}$ in its mapping cone by $Sq^{2n+3}x_{2n+3}=x_{4n+6}$. But, this is impossible as $Sq^{2n+3}$ decomposes as a composition of $Sq^{2n+2}$ and $Sq^1$. This completes the proof.
\end{proof}

Note that in the last paragraph of the above proof, we could have appealed to Adams Hopf invariant one result, but we chose this at the proof seems to make more use of the fact that our maps are unstable maps.

\subsection{Spherical classes in $\Omega^3S^{n+3}$}
Note that our main theorem determines spherical classes in $H_*\Omega^3S^{n+3}$ for $n\leqslant 10$. Here, we deal with the remaining cases.
\begin{thm}\label{tripleloop}
Suppose $f:S^i\to \Omega^3S^{n+3}$ with $h(f)\neq 0$. Then $i=n$ and $f$ is given with the inclusion of the bottom cell.
\end{thm}

Similar to previous cases, for obvious reasons, we only have to deal with the case $i>n$.

\begin{lmm}\label{tripleloop-1}
Suppose $f:S^i\to \Omega^3S^{n+3}$ with $h(f)\neq 0$ and $i>n$ and $n\geqslant 6$. Then, we have one of the following cases:\\
(i) $h(f)=(Q^{2n+3}Q^{n+2}x_n)^2$;\\
(ii) $h(f)=(Q^{n+1}x_n)^2$.
\end{lmm}

\begin{proof}
Recall that
$$H_*\Omega^3S^{n+3}\simeq\Z/2[Q_Jx_n:0\leqslant j_1\leqslant\cdots\leqslant j_s\leqslant 2].$$
Theorem \ref{doubleloop} together with Lemma \ref{2^t} imply that $h(f)=\xi^2$ for some odd dimensional class $\xi$. As previous, by dimensional reasons together with the fact that $\xi$ is primitive, we may write $\xi$ as a sum of terms $Q_Jx_n$ involving no decomposable terms. Similar to the proof of Lemma \ref{doubleloop-1}, this means that $J$ cannot being with $0$, that is $J$ is only a sequence of $1$'s and $2$'s. Suppose
$$J=(\overbrace{1,1,\ldots,1}^{s-\textrm{times}},\overbrace{2,\ldots,2}^{t-\textrm{times}}).$$
Then, we compute that written in terms of upper indexed operations,
$$Q_Jx_n=\underbrace{\cdots Q^{2^{t+1}n+2^{t+2}-2}Q^{2^tn+2^{t+1}-1}}_{\textrm{corresponding to $1$'s in $J$}}\underbrace{Q^{2^{t-1}(n+2)}\cdots Q^{2(n+1)}Q^{n+2}}_{\textrm{corresponding to $2$'s in $J$}}x_n.$$
We note that $J$ is uniquely determined by $s$ and $t$. If $s\geqslant 2$ then we have $Sq^1_*Q_Jx_n\neq 0$ and consequently, $Sq^2_*h(f)\neq 0$. This implies that $s\leqslant 1$. Similarly, if $t>1$ then by Nishida relations, $Sq^{2^{s+2}}_*Q_Jx_n\neq 0$ which shows that $Sq^{2^{s+3}}_*h(f)\neq 0$. This is a contradiction, showing that $t\leqslant 1$. We then have following choices for $\xi$:\\
$(1)$ $\xi=Q_1Q_2x_n=Q^{2n+3}Q^{n+2}x_n$;\\
$(2)$ $\xi=Q_1x_n=Q^{n+1}x_n$;\\
$(3)$ $\xi=Q_2x_n=Q^{n+2}x_n$.\\
However, as $Q_2x_n$ is of the even dimension $2n+2$, we cannot have this case by Lemma \ref{2^t}. Therefore, we are only left with two cases as claimed.
\end{proof}

It remains to eliminate the possibilities in the above lemma from being spherical.
\begin{lmm}\label{tripleloop-2}
Suppose $f:S^i\to \Omega^3S^{n+3}$ with $h(f)\neq 0$, $i>n$ and $n\geqslant 6$. Then, it is impossible to have $h(f)=(Q^{n+1}x_n)^2$.
\end{lmm}

\begin{proof}
First, note that if $h(f)=(Q_1x_n)^2=(Q^{n+1}x_n)^2$ then $i=4n+2$, i.e. $f\in\pi_{4n+2}\Omega^3S^{n+3}\simeq\pi_{4n+5}S^{n+3}$. Consider the James maps $H:\Omega S^{n+3}\to \Omega S^{2n+5}$ and $H:\Omega S^{2n+5}\to \Omega S^{4n+9}$. By viewing $f$ as $f:S^{4n+4}\to\Omega S^{n+3}$, we have $H(f):S^{4n+4}\to \Omega S^{2n+5}$. Another application of $H:\Omega S^{2n+5}\to \Omega S^{4n+9}$ yields $H(H(f)):S^{4n+4}\to\Omega S^{4n+9}$ which is an element of $\pi_{4n+5}S^{4n+9}\simeq 0$. By considering the James fibration $S^{2n+4}\to \Omega S^{2n+5}\to \Omega S^{4n+9}$, this means that $H(f)$ pulls back to an element of $\pi_{4n+4}S^{2n+4}$. Hence, $(\Omega^2 H)\circ f:S^{4n+2}\to\Omega^3S^{2n+5}$ pulls back to a map $S^{4n+2}\Omega^2S^{2n+4}$. By Theorem \ref{doubleloop} this map is trivial in homology, hence the map $\Omega^2 H\circ f:S^{4n+2}\to \to\Omega^2S^{2n+5}$ is trivial in homology.\\
On the other hand, the map $\Omega^2 H:\Omega^3S^{n+3}\to\Omega^3S^{2n+5}$ is a double loop map, hence in homology respects the product as well as the operations coming from the double loop space structure, i.e. $Q_0$ and $Q_1$. In particular, this implies that
$$h(\Omega^2H\circ f)=(\Omega^2H)_*h(f)=(\Omega^2 H)_*(Q_1x_n)^2=(Q_1x_n)^2$$
which means that $h(\Omega^2\circ f)\neq 0$ as $H_*\Omega^3S^{2n+5}$ is polynomial. This contradicts the above observation on triviality of the homology of this map. This completes the proof.
\end{proof}

Next, we prove the last possibility in two steps. First one is quite straightforward.

\begin{lmm}\label{tripleloop-5}
Suppose $n$ is even with $n\geqslant 6$. Then, it is impossible to have $f:S^i\to\Omega^3S^{n+3}$ with $h(f)=(Q^{2n+3}Q^{n+2}x_n)^2$.
\end{lmm}

\begin{proof}
Suppose there exists such a map. By Nishida relations we compute that $Sq^2_*Q^{2n+3}Q^{n+2}x_n=Q^{2n+3}Sq^1_*Q^{n+2}x_n=Q^{2n+2}Q^{n+1}x_n\neq 0$. Consequently, by Cartan formula
$$Sq^4_*h(f)=Sq^4_*(Q^{2n+3}Q^{n+2}x_n)^2=(Sq^2_*Q^{2n+3}Q^{n+2}x_n)^2=(Q^{2n+2}Q^{n+1}x_n)^2\neq 0$$
which is a contradiction.
\end{proof}

In order to eliminate the cases of $h(f)=(Q^{2n+3}Q^{n+2}x_n)^2$ with $n$ odd we need some more numerical tools. For $m$ a nonnegative integer, let $n=\sum_{i=0}^{+\infty} n_i2^i$ with $n_i\in\{0,1\}$ denote its binary expansion. Define $\rho(m)=\min\{i:m_i=0\}$. By looking at the binary expansions, it is easy to see that $\rho(m)$ is the least integer $t$ with
$${n-t\choose t}\equiv 1\textrm{ mod }2.$$
Moreover, note that for ${a\choose b}\equiv 1$ modulo $2$ we need $a_i\leqslant b_i$ in their binary expansions. These observations are useful while evaluating the coefficients in the Nishida relations. We record the following.



We also need to recall description of $H_*Q_0S^0$ in terms of upper indexed operations. Let $\zeta=Q^i[1]*[-2]\in H_iQ_0S^0$ where $[n]$ is the image of $n$ under the Hurewicz map $\pi_0QS^0\to H_0QS^0$. It is known that (see for example \cite[Chapter 1,Section 5]{Snaith-book})
$$H_*Q_0S^0\simeq\Z/2[Q^I\zeta_i:(I,i) \textrm{ admissible },\ex(I)>i]$$
so that $\sigma_*Q^I\zeta_i=Q^IQ^ig_1$ where $\sigma_*^k:H_*Q_0S^0\to H_{*+k}QS^k$ is the $k$-fold homology suspension. The following is due to Curtis \cite[Lemma 6.2]{Curtis} and later on was generalised to odd primes by Wellington \cite[Theorem 11.25]{Wellington}.
\begin{thm}
For $I=(i_1,\ldots,i_s)$, $\ex(Q^I\zeta_i):=i_1-(i_2+\cdots+i_s+i)$. Then, $Sq^t_*Q^Ix_i=0$ for all $t>0$ if and only if\\
(1) $\ex(Q^I\zeta_i)<2^{\rho(i_1)}$ and (2) $2i_{j+1}-i_j<2^{\rho(i_j+1)}$ for $j=1,\ldots,s$ setting $i_{s+1}=i$.
\end{thm}

The above lemma will be proved by first stablising the map $f$, and then adjointing down. This will allow to eliminate $f$ to be $n+2=2^t-1$. We then use some elementary facts on binary expansions to finish off the proof.

\begin{lmm}\label{tripleloop-3}
Suppose $n$ is odd with $n\geqslant 6$. Then, it is impossible to have $f:S^i\to\Omega^3S^{n+3}$ with $h(f)=(Q^{2n+3}Q^{n+2}x_n)^2$.
\end{lmm}

\begin{proof}
Suppose there exists such a map. Note that in this case $i=8n+10$. We claim that\\
\textbf{Claim 1.} $n+2=2^t-1$ for some $t>0$.\\
\textbf{Proof of Claim 1.} Since the stablisation map $E^\infty:\Omega^3S^{n+3}\to QS^n$ is a monomorphism in homology, it follows that for $E^\infty f:S^{8n+10}\to QS^n$ we have $h(E^\infty f)=Q^{4n+5}Q^{2n+3}Q^{n+2}x_n$. Let's write $f^s_{n-k}:S^{i-k}\to QS^{n-k}$ when we adjoint $E^\infty f$ down $k$-times. In particular
$$h(f_0^s)=Q^{4n+5}Q^{2n+3}\zeta_{n+2}+\sum \epsilon_{(K,k)}Q^K\zeta_k+D$$
where $(K,k)$ runs over all admissible sequences $(K,k)$ with $\ex((K,k))<\ex((I,i))$, $\epsilon_{(K,k)}\in\Z/2$, and $D$ is a sum of decomposable terms. The fact that $f_0^s$ is adjoint to $f_n^s=E^\infty f$ implies that
$$\sigma_*^nh(f_0^s)=h(f_0^s)=h(E^\infty f)=Q^{4n+5}Q^{2n+3}Q^{n+2}x_n$$
which means that all terms $Q^K\zeta_k$ will die under homology suspension before $Q^{4n+5}Q^{2n+3}\zeta_{n+2}$, hence $\ex(I,i)-\ex((K,k))\geqslant 2$. The above form for $h(f_0^s)$ implies that
$$h(f_1^s)=\sigma_*h(f_0^s)=Q^{4n+5}Q^{2n+3}Q^{n+2}x_1+\sum \epsilon_{(K,k)}Q^KQ^kx_1$$
where $f_1^s:S^{7n+9}\to QS^1$ with $n$ is an odd dimensional class, hence none of the terms in the above expression can be square. Next, we apply the stable James-Hopf map $j_2:QS^1\to Q\Sigma P$. By \cite[Theorem 5.9]{Snaith-book} we compute that
$$(j_2)_*h(f_1^s)=Q^{4n+5}Q^{2n+3}\Sigma a_{n+2}+\sum \epsilon_{(K,k)}Q^K\Sigma a_k\neq 0$$
where $Q^{(K,k)}$ is of lower excess and hight than $Q^{(I,i)}$. It is straightforward, to see that applying Nishida relations reduces the excess. For $\rho=\rho(n+2)$, we compute that
$$Sq^{2^{\rho+2}}_*h(j_2\circ f_1^s)=Sq^{2^{\rho+2}}_*(j_2)_*h(f_1^s)=Sq^{2^{\rho+2}}_*(Q^{4n+5}Q^{2n+3}\Sigma a_{n+2}+\sum \epsilon_{(K,k)}Q^K\Sigma a_k).$$
The right side of the above equation is given by
$$Q^{4n+5-2^{\rho+1}}Q^{2n+3-2^{\rho}}\Sigma a_{n+2-2^\rho}+O$$
where $O$ is a sum of terms of lower excess than $\ex(4n+5-2^{\rho+1},2n+3-2^{\rho})$, or terms with $Q^L\Sigma a_{n+2}$; terms of second form obviously cannot cancel as they have $a_{n+2}\neq a_{n+2-2^\rho}$. In particular, this implies that $Sq^{2^{\rho+2}}_*h(j_2\circ f_1^s)\neq 0$ which is a contradiction to the fact that $h(j_2E^\infty f)$ is a nontrivial spherical class. Hence, $n+2=2^t-1$ for some $t>0$. This prove Claim 1.\\
Second we claim that \\
\textbf{Claim 2}. It is impossible to have $n+2=2^t-1$.\\
\textbf{Proof of Claim 2.} For $n+2=2^t-1$, $n=2^t-3$, consequently, $2n+3=2^{t+1}-3$ and $4n+5=2^{t+2}-7$ and by looking at the binary expansion we compute that $$\rho(4n+5)=\rho(2n+3)=1.$$
We have
$$h(f)=Q^{2^{t+2}-7}Q^{2^{t+1}-3}Q^{2^t-1}x_{2^t-3}=(Q^{2^{t+1}-3}Q^{2^t-1}x_{2^t-3})^2.$$
Consequently, the adjoint mapping $f_{n-2}:S^{8n+8}\to\Omega^5S^{(n-2)+5}$ satisfies
$$h(f_{n-2})=Q^{2^{t+2}-7}Q^{2^{t+1}-3}Q^{2^t-1}x_{2^t-5}+P^2$$
where $P=\sum Q^Lx_{2^t-5}$ with $\dim(Q^Lx_{2^t-5})=\dim(Q^{2^{t+1}-3}Q^{2^t-1}x_{2^t-5})$ is a sum of primitive classes with $Q^Lx_{2^t-5})\neq Q^{2^{t+1}-3}Q^{2^t-1}x_{2^t-5}$. We examine $Sq^2_*$ which by Nishida relations yields
$$Sq^2_*h(f_{n-2})=Q^{2^{t+2}-9}Q^{2^{t+1}-3}Q^{2^t-1}x_{2^t-4}+Sq^2_*P^2=(Q^{2^{t+1}-3}Q^{2^t-1}x_{2^t-5}+Sq^1_*P)^2.$$
Now, terms in $P$ are of the following types.\\
(1) $l(L)\neq 2$. Since, iterated application of Nishida relations and Adem relations preserve length, hence for such terms $Sq^1_*Q^Lx_{2^t-5}\neq Q^{2^{t+1}-3}Q^{2^t-1}x_{2^t-5}$.\\
(2) $l(L)=2$. In this case, since we know the dimension, the the first entry of $L=(l_1,l_2)$, ie, $l_1$ completely determines $L$. Since, $L\neq ({2^{t+1}-3},{2^t-1})$, then either first or second entry are different. In any case, depending on the parity of $l_1$, we have $Sq^1_*Q^Lx_{2^t-5}=0$ or $Sq^1_*Q^Lx_{2^t-5}=Q^{l_1-1}Q^{l_2}x_{2^t-5}\neq Q^{2^{t+1}-3}Q^{2^t-1}x_{2^t-5}$.\\
These observations, show that the inside bracket in the above equation is nonzero, and consequently, $Sq^2_*h(f_{n-2})\neq 0$. But, this is a contradiction. This completes the proof.
\end{proof}

Finally, we deal with the remaining case of $\Omega^3S^5$.
\begin{thm}\label{tripleloop-4}
The only spherical classes in $H_*\Omega^3S^5$ arise in dimensions $2$ and $5$ corresponding to the inclusion of the bottom cell and $\nu$.
\end{thm}

\begin{proof}
The inclusion of the bottom cell is evident as usual. For $\nu$, we know by Theorem \ref{main} that $\nu:S^5\to\Omega^2S^4$ is nontrivial in homology. Consequently, the composition $S^5\to\Omega^2S^4\to\Omega^3S^5$ is also nontrivial in homology.\\
Suppose $i\neq 2,5$ and $f:S^{i}\to\Omega^3S^5$. Let $f':S^6\to\Omega^2S^5$ be its adjoint. By Theorem \ref{main} the only spherical classes in $H_*\Omega^2S^5$ arise from the inclusion of the bottom cell and $\nu$, in dimensions $3$ and $6$ respectively which we have included. Consequently, $\sigma_*h(f)=h(f')=0$ and $h(f)$ has to be a square. By Lemma \ref{2^t} $h(f)=\xi^2$ for some primitive class in odd dimensions. The results of Lemmata \ref{tripleloop-1}, \ref{tripleloop-2} and \ref{tripleloop-5} hold in this case, which together with the fact that $n=2$ is even imply that $h(f)=0$.
\end{proof}

The elimination of spherical classes in $H_*\Omega^3S^9$ is done in a similar way and we leave it to the reader.

\section{On the relation between Curtis and Eccles conjectures}
This part is mostly expository and well known. We wish to record some observations on the relations between two conjectures. Below, we shall write $P$ for the infinite dimensional real projective space, $P^n$ for the $n$-dimensional real projective space, and $P_n=P/P^{n-1}$ for the truncated projective space with its bottom cell at dimension $n$. We also write $X_n$ for the truncated of a $CW$-complex where all skeleta of dimension $<n$ are collapsed to a point, i.e. $X_n=X/X^n$ where $X^n$ is the $n$-skeleton of $X$.

\begin{lmm}
Curtis conjecture, implies Eccles conjecture for $S^k$ with $k>0$.
\end{lmm}

\begin{proof}
Suppose $f:S^{n+k}\to QS^k$ is given with $h(f)\neq 0$. By adjointing down $k$ times, we have a map, say $f':S^n\to QS^0$, so that $\sigma_*^kh(f')=h(f)$ where $\sigma_*^k:H_*Q_0S^0\to H_{*+k}QS^k$ is the $k$-fold iterated suspension homomorphism. In particular, $h(f')\neq 0$ in $H_*Q_0S^0$. Assuming Curtis conjecture, $f'$ must be either a Hopf invariant or a Kervaire invariant one element. It is well known (see for example \cite[Theorem 7.3]{Madsenthesis}) that the unstable Hurewicz image of a Kervaire invariant one element, if it exists, in $H_*Q_0S^0$ is square of a certain primitive class, say $p_{2^i-1}^2$. However, decomposable classes are killed by homology suspension. So, $f'$ and consequently $f$, as elements of ${_2\pi_*^s}$ can only be Hopf invariant one elements, that is detected by a primary operation in its mapping cone. This prove Eccles conjecture for $S^k$.
\end{proof}

On the other hand, Curtis conjecture can be deduced from Eccles conjecture, thanks to Kahn-Priddy theorem, and its algebraic version due to Lin.

\begin{lmm}
Eccles conjecture for $X=P$ implies Curtis conjecture.
\end{lmm}

\begin{proof}
Consider the Kahn-Priddy map $\lambda:QP\to Q_0S^0$ which is an infinite loop map, inducing an epimorphism on ${_2\pi_*}$ on positive degrees \cite[Theorem 3.1]{Kahn-Priddy}, as well as an epimorphism on the level of Adams spectral sequences $\ext_A^{s,t}(H^*P,\Z/2)\to \ext_A^{s+1,t+1}(\Z/2,\Z/2)$ where $A$ denotes the mod $2$ Steenrod algebra \cite[Theorem 1.1]{Lin}. Suppose $f\in{_2\pi_n}Q_0S^0$ with $h(f)\neq 0$. Let $g\in{_2\pi_n}QP$ be any pull back of $f$ through $\lambda$. Then $g$ maps nontrivially under the unstable Hurewicz map ${_2\pi_n}QP\to H_nQP$. Assuming Eccles conjecture, implies that the stable adjoint of $g$ is detected either by homology or a primary operation in its mapping cone.\\
If the stable adjoint of $g$ is detected by homology, then it is detected on the $0$-line of the Adams spectral sequence for $P$. By Lin's result, $f=\lambda g$ is detected on the $1$-line of the Adams spectral sequence. The $1$-line of the Adams spectral sequence is known to detect Hopf invariant one elements, i.e. the stable adjoint of $f$ is a Hopf invariant one element.\\
Next, suppose the stable adjoint of $g$ is detected by a primary operation in its mapping cone. That is for some $i,j$ we have $Sq^ia_j=x_{n+1}$ in $C_{g'}$ where we write $g':S^n\to P$ for the stable adjoint of $g$. From the action of Steenrod algebra on $H^*P$ and decomposition of Steenrod squares to operations of the form $Sq^{2^t}$, for such an equation to hold in $C_{g'}$ we need $i=2^s$ and $j=2^t$ for some $s,t\geqslant 0$. From this, we see that $g'$ has to be detected in the $1$-line of the Adams spectra sequence for $P$. By Lin's theorem, this means that the stable adjoint of $f$ has to be detected in the $2$-line of the Adams spectral sequence. It is known \cite{Z-decomposables} (see also \cite{Za-ideal} for more details) that the only elements on the $2$-line of the Adams spectral sequence that map nontrivially under $h$ are the Kervaire invariant one element, i.e. $f$ can only be a Kervaire invariant one element. This completes the proof.
\end{proof}


\end{document}